\documentclass[11pt]{amsart}

\usepackage[OT2,T1]{fontenc}
\DeclareSymbolFont{cyrletters}{OT2}{wncyr}{m}{n}
\DeclareMathSymbol{\Sha}{\mathalpha}{cyrletters}{"58}
\usepackage[utf8]{inputenc}
\usepackage{musicography}
\usepackage{amsmath}

\usepackage{array}
\usepackage{accents}

\usepackage{tikz-cd}
\usepackage{hyperref}
\usepackage{verbatim}
\usepackage{amssymb}
\usepackage{textcomp}
\usepackage{amsthm}
\usepackage{amsfonts}
\usepackage[all]{xy}
\usepackage{mathrsfs}
 \usepackage{graphicx}
 \usepackage{epstopdf}
 \usepackage{float}
 \usepackage{graphicx}
 \usepackage[T1]{fontenc}
\newtheorem{theo}{Th\'eor\'eme}[section]
\usepackage[ddmmyyyy]{datetime}

\newcolumntype{L}[1]{>{\raggedright\arraybackslash}m{#1}}

\def\line{\overline}

\def\remk{\noindent\textit{Remark:~}}
\numberwithin{equation}{section}

\newtheorem{def-prop}[theo]{Definition-Proposition}

\newtheorem{lemma}[theo]{Lemma}
\newtheorem{teo}[theo]{Theorem}
\newtheorem{definition}[theo]{Definition}

\newtheorem{remark}[theo]{Remark}

\newtheorem{assumption}[theo]{Assumption}

\def\C{\mathbb{C}}

\def\G{\mathbb{G}}

\def\N{\mathbb{N}}

\def\F{\mathbb{F}}

\def\Q{\mathbb{Q}}

\def\Z{\mathbb{Z}}

\def\kg{\mathfrak{g}}

\def\h{\mathfrak{h}}

\def\k{\mathrm{k}}

  \newcommand{\cC}{{\mathcal C}}
    
  \newcommand{\cD}{{\mathcal D}}

 \def\cO{\mathcal{O}}

\newcommand{\cT}{{\mathcal T}}

\newcommand{\sO}{{\mathscr O}}

\newcommand{\sS}{{\mathscr S}}

 \def\vol {\mathop{\mathrm{vol}}\nolimits}

 \def\ad {\mathop{\mathrm{ad}}\nolimits}

  \def\Orb {\mathop{\mathrm{Orb}}\nolimits}

 \def\v {\mathop{\mathrm{v}}\nolimits}

\def\ker{\mathop{\mathrm{ker}}\nolimits}

\def\dim{\mathop{\mathrm{dim}}\nolimits}
\def\Tr{\mathop{\mathrm{Tr}}\nolimits}

\def\rank{\mathop{\mathrm{rank}}\nolimits}

\def\im{\mathop{\mathrm{Im}}\nolimits}

\def\mod{\mathop{\mathrm{mod}}\nolimits}

\def\GL{\mathop{\mathrm{GL}}\nolimits}
 
 \def\SL {\mathop{\mathrm{SL}}\nolimits}
 
 \def\sl{\mathop{\mathrm{sl}}\nolimits}
\def\gl{\mathop{\mathrm{gl}}\nolimits}

\def\Aut{\mathop{\mathrm{Aut}}\nolimits}

\def\Id{\mathop{\mathrm{Id}}\nolimits}
\def\hat{\widehat}
\def\Ad{\mathop{\mathrm{Ad}}\nolimits}

\allowdisplaybreaks

\everymath{\displaystyle}

\newlength{\dhatheight}

\makeatletter
\renewcommand\section{\@startsection{section}{1}{\z@}%
                       {-18\p@ \@plus -4\p@ \@minus -4\p@}%
                       {12\p@ \@plus 4\p@ \@minus 4\p@}%
                       {\normalfont\large\bfseries\boldmath
                        \rightskip=\z@ \@plus 8em\pretolerance=10000 }}
\renewcommand\subsection{\@startsection{subsection}{2}{\z@}%
                       {-18\p@ \@plus -4\p@ \@minus -4\p@}%
                       {8\p@ \@plus 4\p@ \@minus 4\p@}%
                       {\normalfont\normalsize\bfseries\boldmath
                        \rightskip=\z@ \@plus 8em\pretolerance=10000 }}
\renewcommand\subsubsection{\@startsection{subsubsection}{3}{\z@}%
                       {-18\p@ \@plus -4\p@ \@minus -4\p@}%
                       {4\p@ \@plus 2\p@ \@minus 2\p@}%
                       {\normalfont\normalsize\bfseries\boldmath
                        \rightskip=\z@ \@plus 8em\pretolerance=10000 }}
\renewcommand\paragraph{\@startsection{paragraph}{4}{\z@}%
                       {-12\p@ \@plus -4\p@ \@minus -4\p@}%
                       {2\p@ \@plus 1\p@ \@minus 1\p@}%
                       {\normalfont\normalsize\itshape
                        \rightskip=\z@ \@plus 8em\pretolerance=10000 }}
\makeatother

\let\op\operatorname

\input xy
\xyoption{all}

\title{On generalization of a theorem of Harish-Chandra}
\author{Taiwang DENG}
\address{ 
Yanqi Lake Beijing Institute of Mathematical Sciences and Applications (BIMSA), Huairou District, 100084, Beijing, China.}
\email{dengtaiw@bimsa.cn}
\date{}


\begin{document}

\maketitle

\begin{abstract}
In this paper under some conditions we generalize a theorem of Harish-Chandra concerning representability of Fourier transforms of 
orbital integrals.
\end{abstract}

\section{Introduction}
In this short paper we consider the question of representability of Fourier transform of orbital integrals for general representations of a reductive groups $G$.
In the case $\rho: G\rightarrow \Aut(\kg)$ is the adjoint representation, let
\[
\hat{\Orb_G}(X, f):=\Orb_G(X, \hat{f}), f\in \cC_c^\infty(\kg), X\in \kg,
\]
where $\Orb_G(X, f)$ is the orbital integral of $f$ along  the $G$-orbit of $X$, then Harish-Chandra \cite[Theorem 1.1]{De99} shows that the distribution
is representable by a kernel function $\kappa(X, \cdot)$. He approaches the problem by local trace formula and a finiteness theorem due to Howe.
In \cite[\S 4]{Zhang12}, Zhang raises the question whether analogues theorem holds for the general representations of $G$. Note
that Howe's finiteness theorem is known to fail in general (see loc.cit. for counter examples). In this paper, we answer this question 
affirmatively under some mild conditions. Our main result is Theorem \ref{teo-main-results-kernel-function}. After we finish the proof, 
we notice that a particular case when $G=\SL_2$ and $V=\sl_2$ the adjoint representation is studied in \cite{Eve98}, where the author applies the same method to obtain almost the same result.

Let $F$ be a finite extension of $\Q_p$.
Let $\cO_F$ be the ring of integers of $F$ and $\pi$ be its uniformizer. Let $\k_F$ be its residue field.
We fix a valuation $\v_{\pi}$ on $F$ with $\v_{\pi}(\pi)=1$.
Let us also fix an additive character $\psi$ of $F$ satisfying
\[
\psi|_{\cO_F}=1, \quad \sum_{a\in \pi^{-1}\cO_F/\cO_F}\psi(a)=0.
\]
Assume that $G$ is a quasi-split reductive group defined over  $\cO_F$.
In general, $G\otimes F$ may have different models over 
$\cO_F$, and the reductive models depend on a choice of a special vertex on the Bruhat-Tits building of $G\otimes F$. Note that
it follows from our assumption that the group $G\otimes \k_F$ is also reductive.

Let us fix the notations. Assume that $(\rho, V)$ is an algebraic representation of a reductive group $G$ over 
 $\cO_F$. We identify $\G_m$ as the center of $\Aut(V)$. 
 
We assume that  $G$ is a reductive subgroup such that $G\cap \G_m$ is a finite subgroup 
and denote $H=G\G_m$.
 Finally we denote by $\h$ the Lie algebra of $H$ and $\kg$ the Lie algebra of $G$.

According to \cite[\S 3 and \S 7]{Pra15}, there exists an involution $\theta$ of $G$
defined over $F$(called Chevalley involution) which interchanges an irreducible and its contragredient.

The existence of the Chevalley involution allows us to construct a non-degenerate symmetric pairing on $V_F=V\otimes F$
(without confusion we identify $V_F$ with its $F$-rational points $V(F)$)
\[
\langle ,\rangle: V_F\otimes V_F\rightarrow F
\]
satisfying
\[
\langle gX , \theta(g)Y\rangle=\langle X , Y\rangle, \quad \forall X, Y\in V(F), g\in G(F).
\]
We also note that $\theta$ can be extended to an involution on $H$.
\begin{definition}\label{def-valuation-norm}
We fix a basis $\{e_1, \cdots, e_d\}$ of $V(\cO_F)$. Define a valuation $\v_{\pi}$ on $V$ by letting
\[
\v_{\pi}(\sum_i x_i e_i)=\min\{\v_{\pi}(x_i): i=1, \cdots, d\}.
\]
We also define the norm $|Y|=|\k_F|^{-\v_{\pi}(Y)}$.
For any subspace $V_1$ of $V(F)$,  define
\[
d_{\langle ,\rangle}(V_1)=\min\{\v_{\pi}(Y)| Y\in V_1, \langle Z, Y \rangle\in \cO_F, \forall Z\in V(\cO_F)\}.
\]
Finally, define the depth of $\langle ,\rangle$ by
\[
d_{\langle ,\rangle}:=\max\{d_{\langle ,\rangle}(V_1)| V_1\subseteq V\}.
\]
\end{definition}
It follows immediately that 
\begin{lemma}
The valuation $\v_{\pi}$ is independent of the choice of the basis of $V(\cO_F)$.
\end{lemma}

\begin{lemma}\label{lemma-depth-imply-integrality}
For any $Y\in V(F)$ with $\v_{\pi}(Y)\geq d_{\langle, \rangle}$, we have
\[
 \langle Z, Y \rangle\in \cO_F, \forall Z\in V(\cO_F).
\]
\end{lemma}

\begin{proof}
In fact, let $Y\in V(F)$ with $\v_{\pi}(Y)\geq d_{\langle, \rangle}$ and $Z\in V(\cO_F)$ such that 
\[
 \langle Z, Y \rangle\notin \cO_F.
\]
Then take $V_1$ to be the subspace generated by $Y$. Then $d_{\langle ,\rangle}(V_1)\geq \v_{\pi}(Y)\geq d_{\langle, \rangle}$, 
which is absurd. 
\end{proof}

\begin{assumption}\label{ass-choice-depth-pairing}
We assume that our choice of $\langle ,\rangle$ is of depth zero.
\end{assumption}

\begin{remark}  In general if both $G$ and $V$ are defined over $\Z$, then for fixed $\langle-,-\rangle$, after excluding finitely many primes the above assumption
is always satisfied. In case $G=\SL_2$ and $V=\sl_2$ is the adjoint representation, we can take 
\[
\langle U, X\rangle: = 2 ux+vz+wy, \quad X=\begin{pmatrix} x&y\\ z&-x\end{pmatrix}, U=\begin{pmatrix} u&v\\ w&-u\end{pmatrix}.
\]
In this case, only $p=2$ violates the above assumption.
\end{remark} 

Let $\cC_c^{\infty}(V)$ be the space of locally constant functions on $V(F)$ with compact support.  For $Y\in V(F)$ and $f\in \cC_c^\infty(V)$, we define the
Fourier transform of $f$ by
\[
\hat{f}(Y)=\int_{V(F)} f(X)\psi(\langle X, Y \rangle)dX.
\]
We consider the orbital integrals of $f$ on $V(F)$ to be a linear functional on $\cC_c^\infty(V)$:
\begin{equation}\label{eqn-definition-orbital-integral}
\Orb_G(X, f):=\int_{\sO_G(X)}f(Y)dY,
\end{equation}
where $\sO_G(X)$ denote the $G$-orbit of $X$ in $V(F)$.
We denote by $\sS(V)$ the linear functional on $\cC_c^\infty(V)$ generated by the orbital integrals.
We are interested in the question of representability of the orbital integral $\Orb_G(X, \hat{f})$, cf. \cite[\S 3]{Zhang12}.

In general, it is convenient to introduce a variant of the orbital integrals, especially when we deal with the orbital integrals of $H$. 
We assume that there exists a non-constant H-quasi-invariant $P(X)$  in the polynomial ring $F[V]$  
such that the set 
\[
V_P:=\{X\in V(F)| P(X)\neq 0\}
\]
is open and dense in $V(F)$( in the $p$-adic topology). By assumption, we have
\begin{equation}\label{eqn-disc-equation}
(g.P)(X)=P(g^{-1}X)=\nu^{-1}(g)P(X), \forall g\in H, 
\end{equation}
for some algebraic character $\nu: H\rightarrow \G_m$. Let $\nu_s(g)=|\nu(g)|^s$.

Let $\chi: F^\times \rightarrow \C^\times$ be a (unitary) character. The orbit $\sO_H(X)$ of $X\in V(F)$ admits an action of $F^\times$ identifying with the center of $\GL(V)$, which commutes with 
the action of $G$.
We modify our definition of orbital integral to be 
\begin{align}
\Orb_G(X, f, \nu_s):&=\int_{\sO_G(X)}f(Y)|P(Y)|^s dY \label{eqn-orbital-integral-definition-modification_1},\\
\Orb_H(X, f, \chi\nu_s):&=\int_{F^\times} \chi(t) \Orb_G(tX, f, \nu_s) d^\times t\label{eqn-orbital-integral-definition-modification_2},
\end{align}
where $dY$ is a quasi-invariant algebraic measure on $\sO_G(X)$ coming from the algebraic volume form such that 
\[
d(t Y)= \delta( t) dY, t\in F^\times
\]
with $\delta$ a (quasi) character of $F^\times$.
 Observe that when $\chi=1$, the term $\Orb_H(X, f, \chi\nu_s)$ can be rewritten as
\begin{equation}
\Orb_H(X, f, \nu_s)=\int_{\sO_H(X)}f(Y)|P(Y)|^s dY, 
\end{equation}
which is compatible with (\ref{eqn-orbital-integral-definition-modification_1}). 

Let $\cO_H^{\op{Zar}}$ be a geometric orbit of $H$ in $V$ which is defined over $F$. It is an affine variety since $H$ is reductive.

We impose some assumptions on the orbit $\cO_G(X)$.
\begin{assumption}\label{ass-compact-condition}
 We assume that $\cO_G(X)\subseteq V_P$ is closed.
\end{assumption}

Note that we do not ask the orbit $\cO_G(X)$ to be closed in $V$ since it will rule out our example of prehomogeneous space (in which we have $\cO_G(X)=V_P$).
This assumption is crucial for certain compactness result (cf. Lemma \ref{lem-compactness-result}).

Put $q=|\k_F|$. For $n\in \Z$, set $W_n(X)=\pi^{-n}V(\cO_F)\cap \sO_G(X)$.
\begin{assumption}\label{ass-regularity-condition}
 We assume that for any $\epsilon>0$ and $n\in \Z$,
\[
\int_{\cO_H(X)}1_{\{Y\in  W_n(X):  |P(Y)|>\epsilon \}}(Y)dY<\infty.
\]
\end{assumption}

Note that this assumption excludes the case of taking $P(Y)$ to be the constant function in general.
This leads to the following lemma

\begin{lemma}\label{lem-rationality-p-adic-integral}
Under the assumption \ref{ass-regularity-condition}
the integral (\ref{eqn-orbital-integral-definition-modification_1})  converges for $\op{Re}{s}\gg 0$.  Moreover, there exists a meromorphic function
\[
M(s)=\frac{Q(q^{-s}, q^s)}{(1- q^{-a_1-b_1s})^{n_1}(1- q^{-a_2-b_2s})^{n_2}\cdots (1- q^{-a_r-b_rs})^{n_r}}
\]
on $\C$  and
\begin{itemize}
\item $r\in \N$, $ (a_1, b_1), (a_2,b_2), \cdots, (a_r, b_r)$ are pairwise distinct integers such that $b_i\geq 0 (i=1,\cdots, r)$ and $b_i>0$ if $a_i<0$;
\item the $b_i$'s depend only on $P(X)$;
\item $n_1, \cdots, n_r\in \{ 1, 2, \cdots, \dim(\cO_G(X))\}$;
\item $Q$ is a two variable polynomial with coefficients in $\C$,
\end{itemize}
such that if the integral (\ref{eqn-orbital-integral-definition-modification_1}) is absolutely convergent at $s=s_0$, then $M(s)$ is holomorphic at $s=s_0$ and 
\[
\Orb_G(X, f, \nu_{s_0})=M(s_0).
\]
\end{lemma}

A version of this lemma is the main result of \cite[Theorem 1.1]{Denef84} obtaining by $p$-adic  cell decomposition. Later a new proof is given in \cite[Corollary 15]{CL08} using $p$-adic
integration theory which rely on a partition of semi-algebraic sets into simple pieces(cf. \cite[Theorem 7]{CL08}). The above version we stated is a special case of \cite[Theorem 5.13]{HS17} (with comparison to  \cite[Corollary 15]{CL08} and \cite[Theorem 1.3]{CD08} ).

As for the rationality of $p$-adic orbital integrals,  \cite[Theorem 1.3]{CD08} proves a general version of it for arbitrary linear algebraic group under some complicated conditions.
In special case of $G=\GL_n$ with adjoint representation, Yun \cite[Corollary 4.6]{Yun13} deduce the rationality of $p$-adic orbital integrals by connecting it to counting lattices, which is quite amenable for generalizations. We hope to come back to this point in future work.

\begin{remark} The assumption is satisfied for the case of prehomogeneous space \cite[Lemma 2]{Igu84} as well as the case of $(\rho, V)=(\Ad, \frak{g})$ the adjoint representation.
Here the second case follows from the first as is explained in \cite[\S 3]{AM96}. 
\end{remark}

\vspace{.1cm}

Our main result in this paper is the following

\begin{teo}\label{teo-main-results-kernel-function}
Assume the hypothesis \ref{ass-choice-depth-pairing} and \ref{ass-regularity-condition}.
There exists a locally constant  function $\kappa_G(\nu_s, X, \cdot)$ 
on $V(F)$ given by the principal value integral
\begin{equation}\label{eqn-kernel-function}
\int_{\sO_G(X)}|P(Z)|^s\psi(\langle Z, Y \rangle) d Z, \quad \op{Re}(s)\gg 0.
\end{equation}
It
is a rational function in $q^{-s}$ with order of poles bounded by constant depending only on $P(X)$ and $\dim(\cO_G(X))$.
When $s=s_0$ is not its pole, then the function $\kappa_G(\nu_{s_0}, X, \cdot)$  is locally integrable. 
If furthermore $s=s_0$ is also not the pole of $\Orb_G(X, f, \nu_s)$, then
we have for $f\in \cC_c^{\infty}(V)$,
\begin{align}\label{teo-main-equality-Fourier}
\Orb_G(X, \hat{f}, \nu_{s_0})=\int_{V(F)} f(Y)\kappa_G(\nu_{s_0}, X, Y)dY.
\end{align}
\end{teo}

\begin{remark} 
When $\Orb_G(X, f, \nu_s)$ has no poles at $s=0$,
letting $s=0$ we recover the usual definition (\ref{eqn-definition-orbital-integral}) of orbital integrals for $G$.
\end{remark} 

In fact, we  can also derive from our main results a twisted generalization for $H$, which is the following:

\begin{teo}\label{teo-main-results-kernel-function-H}
Assume the hypothesis \ref{ass-choice-depth-pairing} and \ref{ass-regularity-condition}. There exists a locally constant function $\kappa_H(\chi\nu_s, X, \cdot)$ well defined outside a measure zero subset
on $V(F)$ given by the integral
\[
\int_{F^\times}\chi(t)\kappa_G(\nu_s, tX, \cdot) d^\times t
\]

For $f\in \cC_c^{\infty}(V)$,
\begin{align}\label{teo-main-equality-Fourier-H}
\Orb_H(X, \hat{f}, \chi\nu_s)=\int_{V(F)} f(Y)\kappa_H(\chi \nu_s, X, Y)dY.
\end{align}

\end{teo}

Our theorem \ref{teo-main-results-kernel-function} implies Harish-Chandra’s original theorem, which corresponds to the case
when $G$ is semi-simple and $(\rho, V)$ is the adjoint representation. See the discussion after the proof of the theorem.

\begin{remark}
When $V_F$ contains only finitely many open $H$-orbits $\{\sO_1, \cdots, \sO_r\}$, then the above theorem gives
\[
\Orb_H(\sO_i, \hat{f}, \nu_s)=\sum_{j} \Gamma_{ij}(\nu_s) \Orb_H(\sO_j, f, \theta(\delta^{-1} \nu_{s}))
\]
where $\delta: H(F) \rightarrow \C^\times$ via $d(h_1\line{h})=\delta(h_1)d\line{h}$ and $\theta(\delta^{-1} \nu_{s}))(h)=\delta^{-1} \nu_{s}(\theta(h))$, this is the local functional equation for local Igusa Zeta function \cite[Theorem 1]{Igu84}.
\end{remark}

\section{Proof of the Main Results}

\begin{proof}[Proof of Lemma \ref{lem-rationality-p-adic-integral}]
Consider
\[
W_n(X)=\pi^{-n}V(\cO_F)\cap \sO_G(X)
\]
which is an semi-algebraic set which is open in the semi-algebraic space $\cO_G(X)$ (This is even a Nash manifold, cf. \cite[\S 6.5]{HS17}).
We endow it with the $\delta$-invariant measure $dY$ (or the one from the volume form). When $X\in V_P$, 
we observe that the restriction of $P(X)$ to $W_n(X)$ is a nowhere vanishing and bounded semi-algebraic function, hence 
we can apply the \cite[Theorem 5.13]{HS17}. Note that our choice of measure is of order $\leq 0$ as indicated by \cite[Theorem 6.15]{HS17}.
This explains the special form of  $M(s)$ in our lemma (Compare also with \cite[Corollary]{CL08}).
\end{proof}

\remk The exponents $b_i$'s are missed in \cite[Theorem 5.13]{HS17}, which is clear according to the proof appearing after \cite[Lemma 5.17]{HS17}.

Let us consider the proof of Theorem \ref{teo-main-results-kernel-function}. 
We need some preparations.

\begin{lemma}\label{lemma-non-empty-open-subset}
For $Z_0\in V(F)\backslash \{0\}$, let $V_0$ be the $\cO_F$-module generated by $H(\cO_F).Z_0$. Let 
\[
d_{V_0}=\min\{\v_{\pi}(Y)| Y\in V(F),  \langle Z, Y\rangle\in \cO_F,  \forall Z\in V_0 \}.
\]
Let 
\[
U_{Z_0}=\{Y\in V(F)| \psi(\frac{\langle\cdot,  \pi^{d_{V_0}-\v_{\pi}(Z_0)-\v_{\pi}(Y)}Y\rangle}{\pi}): V_0\otimes \k_F\rightarrow \C^\times \text{ is nontrivial} \}.
\]
Then $U_{Z_0}$ is an non-empty open subset of $V(F)$ whose complement is closed of measure zero.
\end{lemma}

\begin{proof}[Proof of Lemma \ref{lemma-non-empty-open-subset}]
The non-emptiness of $U_{Z_0}$ follows from the definition of $d_{V_0}$.
By definition, if $Y\notin U_{Z_0}$, then $\langle Z, \pi^{d_{V_0}-\v_{\pi}(Y)-\v_{\pi}(Z)}Y\rangle \in \pi\cO_F$ for all $Z\in V_0$.
Let $U^c$ be the complement of $U_{Z_0}$ in $V(F)$ and $U^c_n=U^c\cap \pi^nV(\cO_F)$.
 It suffices to show that for any $n>0$, the image of $U_n^c$ in $\pi^nV(\cO_F)/\pi^{n+1}V(\cO_F)$ is not surjective.
In fact, let $Y\in U_n^c\backslash U_{n+1}^c$, then $\v_{\pi}(Y)=n$, but by definition $Y_0=\pi^{-n}Y\in U_0^c$, hence 
we only need to show that the image of $U_0^c$ in $V(\k_F)$ is not surjective. But it readily follows from the definition
there exists $Y\in V(F), Z\in V_0$ such that $\v_{\pi}(Y)=0$ and $\langle Z, \pi^{d_{V_0}-\v_{\pi}(Z_0)}Y \rangle\in \cO_F\backslash \pi\cO_F$.
Hence the projection of $U_0^c$ is not surjective to $V(\k_F)$.
\end{proof}

\begin{lemma}\label{lemma-summation-vanishing-character}
We keep the notations in Lemma \ref{lemma-non-empty-open-subset}. Then for $Y\in U_{Z_0}$, we have
\[
\sum_{Z\in V_0\otimes \k_F}\psi(\frac{\langle Z,  \pi^{d_{V_0}-\v_{\pi}(Z_0)-\v_{\pi}(Y)}Y\rangle}{\pi})=0.
\]
\end{lemma}
\begin{proof}[Proof of Lemma \ref{lemma-summation-vanishing-character}]
Note that $\psi(\frac{\langle Z,  \pi^{d_{V_0}-\v_{\pi}(Z_0)-\v_{\pi}(Y)}Y\rangle}{\pi})$ is a non-trivial character on $V_0\otimes \k_F$.
But for any nontrivial character $\chi': V_0\otimes \k_F\rightarrow \C^\times$, assume that $\chi'(Z)\neq 1$ and $V_0\otimes \k_F=V'\oplus \k_F Z$, then
\[
\sum_{X\in V_0\otimes \k_F}\chi'(X)=\sum_{Z'\in V'}\chi'(Z')\sum_{a\in \k_F}\chi'(Z)^a=0.
\]
\end{proof}

\begin{lemma}\label{lem-compactness-result}
Let $W_n(X)=\pi^{-n}V(\cO_F)\cap \sO_G(X)$, then
$W_n( X)\backslash W_{n-1}(X)$ is a compact subset of $\sO_G(X)$.
\end{lemma}

\begin{proof}
We regard $V(F)$ as a metric space with respect to the norm in definition \ref{def-valuation-norm}, i.e.,
\[
|Y|:=q^{\v_{\pi}(Y)}.
\]
It follows from the completeness of $F$ that $V(F)$ with respect to this metric is complete. Note that by our assumption $G$ has finite center, hence 
the algebraic character in (\ref{eqn-disc-equation}) restricting to $G$ is trivial, hence for any $g\in G$,
\[
P(gY)=\nu^{-1}(g)P(Y)=P(Y).
\]
Now assume that $g_i\in G(i=1,2,\cdots)$ with $g_i(X)\in W_n( X)\backslash W_{n-1}(X)$ such that 
$Y_0=\lim_{i}g_i(X)$ in $V(F)$. In particular, we have
\[
P(Y_0)=\lim_{i}P(g_i(X))=P(X)\neq 0.
\]
Therefore we have $Y_0\in V_P$. By assumption $\cO_G(X)$ is closed in $V_P$ hence we must have $Y_0\in \cO_G(X)$.
It is also clear that $Y_0\in \pi^{-n}V(\cO_F)\backslash \pi^{-n+1}V(\cO_F)$, which implies that $Y_0\in W_n( X)\backslash W_{n-1}(X)$.
It follows then $W_n( X)\backslash W_{n-1}(X)$ is actually closed in $V$.
Since it is also bounded, it is compact.
\end{proof}

\begin{lemma}\label{lemma-kernel-constancy-well-define}
We consider the sequence of functions 
\[
f_n:=1_{\pi^{-n}V(\cO_F)}, \quad n\in \Z.
\]

And denote for $\rm{Re}(s)>0$,
\[
I_{n}(Y):=\int_{\sO_G(X)}f_n(Z)|P(Z)|^s\psi(\langle Z, Y \rangle) d Z.
\]
Then there exists an open dense subset $U$ of $V(F)$ such that $V(F)\backslash U$ is of measure zero and for any $Y\in U$,
the sequence $\{I_{n}(Y)| n\in \Z\}$ stabilizes as $n\geq \v_{\pi}(Y)+3$. Consequently, 
the sequence $\{I_{n}(Y)| n\in \Z\}$ convereges to a locally constant function $I_{\infty}(Y)$ on $U$. 
Furthermore,  the function $I_\infty(Y)$ 
is a rational function in $q^{-s}$ with order of poles bounded by constant depending only on $P(X)$ and $\dim(\cO_G(X))$.
\end{lemma}

\remk The example below with $H=\GL_n$ shows that the restriction to $U$ is necessary.

\begin{proof}[Proof of Lemma \ref{lemma-kernel-constancy-well-define}]

Let $W_n(X)=\pi^{-n}V(\cO_F)\cap \sO_G(X)$, then
\begin{align*}
I_n(Y)&=\int_{W_n(X)}|P(Z)|^{s}\psi(\langle Z, Y\rangle)dZ\\
         &=\int_{W_n( X)\backslash W_{n-1}(X)}|P(Z)|^{s}\psi(\langle Z, Y\rangle)dZ+ \int_{W_{n-1}(X)}|P(Z)|^{s}\psi(\langle Z, Y\rangle)dZ\\
         &=I_{n-1}(Y)+\int_{W_n( X)\backslash W_{n-1}(X)}|P(Z)|^{s}\psi(\langle Z, Y\rangle)dZ.
\end{align*}
Therefore, we have
\[
I_n(Y)-I_{n-1}(Y)=\int_{W_n( X)\backslash W_{n-1}(X)}|P(Z)|^{s}\psi(\langle Z, Y\rangle)dZ.
\]
It suffices to show that for $n\geq \v_{\pi}(Y)+3$, 
\begin{equation}\label{eqn-fundamental-vanishing}
\int_{W_n( X)\backslash W_{n-1}(X)}|P(Z)|^{s}\psi(\langle Z, Y\rangle)dZ=0.
\end{equation}
Note that $W_n( X)\backslash W_{n-1}(X)$ lies in the $G(F)$ orbit $\sO_G(X)$ and is a compact subset by Lemma \ref{lem-compactness-result}.

Therefore it decomposes into finitely many $G(\cO_F)$ orbits:
\[
W_n( X)\backslash W_{n-1}(X):=\coprod_{i=1}^r G(\cO_F).Z_i
\]
with $D=\{Z_1, \cdots, Z_r\}$ a set of representatives. Hence it suffices to show that 
\[
\int_{G(\cO_F).Z_i}|P(Z)|^s\psi(\langle Z, Y\rangle) dZ=0, \text{ for  } n\gg0.
\]
Note that for $g\in G(\cO_F)$, we have 
\[
|P(gZ_i)|=|\nu(g)| |P(Z_i)|=|P(Z_i)|
\]
since $\nu$ is an algebraic character. Hence 
\begin{align*}
\int_{G(\cO_F).Z_i}|P(Z)|^s\psi(\langle Z, Y\rangle) dZ&=c_i\sum_{h\in G(\cO_F/\pi^n\cO_F)}\psi(\langle hZ_i, Y\rangle)
\end{align*}
with $$c_i=\frac{\vol(\ker(G(\cO_F)\rightarrow G(\cO_F/\pi^n\cO_F)))|P(Z_i)|^s}{|\im\{G^{i}\rightarrow G(\cO_F/\pi^n\cO_F)\}| },$$ 
where 
$G^{i}$ is the stabilizer of $Z_i$ in $G(\cO_F)$. Assume that $c_i\neq 0$ for some $i$ otherwise we are done. 
Let $Y_1=\pi^{-\v_{\pi}(Y)}Y$. Now for each such $i$,  
\[
\sum_{h\in G(\cO_F/\pi^n\cO_F)}\psi(\langle hZ_i, Y\rangle)=\sum_{h\in G(\cO_F/\pi^n\cO_F)}\psi(\pi^{-\v_{\pi}(Y)}\langle hZ_i, Y_1\rangle).
\]
We claim that 
\begin{equation}\label{eqn-vanishing-summation-character}
\sum_{h\in G(\cO_F/\pi^n\cO_F)}\psi(\langle hZ_i, Y\rangle)=0,\quad \forall n\geq \v_{\pi}(Y)+3.
\end{equation}
Upon replacing $Y$ by $Y_1$ above, we only need to show that for $\v_{\pi}(Y)=0$, 
\[
\sum_{h\in G(\cO_F/\pi^n\cO_F)}\psi(\langle hZ_i,  Y\rangle)=0,\quad \forall n\geq 3.
\]
Recall that by our assumption, for any $Y\in V(\cO_F)$, $\langle Z, Y \rangle \in\cO_F$ holds for any $Z\in V(\cO_F)$(cf. Lemma \ref{lemma-depth-imply-integrality}).
Let $G_n=\ker(G(\cO_F)\rightarrow G(\cO_F/\pi^n\cO_F))$ and $G_{n, n-1}$ the image of $G_{n-1}$ in $G(\cO_F/\pi^n\cO_F)$ under the natural projection.
The key observation is that $G_{n, n-1}$ is a $\k_F$-vector space which is isomorphic to the Lie algebra $\kg(\k_F)$ via $\Id+\pi^{n-1}h\mapsto h$.  
We further identify $G(\cO_F/\pi^{n-1}\cO_F)$ as a subset of $G(\cO_F/\pi^{n}\cO_F)$ by taking a section. Now
\begin{align*}
&\sum_{h\in G(\cO_F/\pi^n\cO_F)}\psi(\langle hZ_i,  Y\rangle)\\
&=\sum_{h_1\in G(\cO_F/\pi^{n-1}\cO_F)}\sum_{h_2\in \kg(\k_F)}\psi(\langle h_1(1+\pi^{n-1}h_2)Z_i,  Y\rangle)\\
&=\sum_{h_1\in G(\cO_F/\pi^{n-1}\cO_F)}\psi(\langle h_1Z_i,  Y\rangle)\sum_{h_2\in \kg(\k_F)}\psi(\langle h_2(\pi^{n-1}h_1^{-1}Z_i),  Y\rangle).
\end{align*}
Note that since $\pi^{n-1}h_1^{-1}Z_i\in W_{1}(X)$, the function $h\mapsto \psi(\langle h(\pi^{n-1}h_1^{-1}Z_i),  Y\rangle$ is well defined on $\kg(\k_F)$.
Now apply Lemma\ref{lemma-summation-vanishing-character}, we know that there is 
an open dense subset $U_{Z_i}$ of $V(F)$ such that 
$V(F)\backslash U_{Z_i}$ is of measure zero and 
$\psi(\langle\cdot,  Y\rangle)$ is non-trivial as a character of $\kg(\k_F)(\pi^{n-1}h_1^{-1}Z_i)$ for each $h_1\in G(\cO_F/\pi^{n-1}\cO_F)$, where 
$\kg(\k_F)(\pi^{n-1}h_1^{-1}Z_i)$ is considered as a subspace of the $\k_F$ vector space generated by $\pi^{-1}V(\cO_F)/V(\cO_F)$.
As a consequence, we have
\[
\sum_{h_2\in \kg(\k_F)}\psi(\langle h_2 (\pi^{n-1}h_1^{-1}Z_i),  Y\rangle)=0, \quad \forall h_1\in G(\cO_F/\pi^{n-1}\cO_F).
\]
Here we are summing over $ \kg(\k_F)$ instead of its subspace $\kg(\k_F)(\pi^{n-1}h_1^{-1}Z_i)$ hence the vanishing of summation over the latter implies
the former.
Let $U=\cap U_{Z_i}$, then  for $Y\in U$, we have 
\[
\int_{W_n( X)\backslash W_{n-1}(X)}|P(Z)|^{s}\psi(\langle Z, Y\rangle)dZ=0.
\]
which finishes the proof of the claim (\ref{eqn-vanishing-summation-character}).  Now for $Y\in U$, we have 
\[
I_\infty(Y)=\sum_{n=\v_{\pi}(Y)+1}^{\v_{\pi}(Y)+2}\int_{W_{n}(X)\backslash W_{n-1}(X)}|P(Z)|^s\psi(\langle Z, Y\rangle) dZ+I_{\v_{\pi}(Y)}(Y).
\]
Furthermore, 
\begin{align*}
I_{\v_{\pi}(Y)}(Y)&=(\nu_s\delta)^{-\v_{\pi}(Y)}(\pi)\int_{W_0(\pi^{\v_{\pi}(Y)}X)}|P(Z)|^s\psi(\frac{\langle Z, Y\rangle}{\pi^{\v_{\pi}(Y)}}) dZ\\
    &=(\nu_s\delta)^{-\v_{\pi}(Y)}(\pi)\int_{W_0(\pi^{\v_{\pi}(Y)}X)}|P(Z)|^s dZ,
\end{align*}
here we use the assumption that $\langle , \rangle$ is of depth zero. Now 
we apply Lemma \ref{lem-rationality-p-adic-integral} to show that
the integral 
\[
\int_{W_0(\pi^{\v_{\pi}(Y)}X)}|P(Z)|^s dZ
\]
is a rational function in $q^{-s}$ with order of poles bounded by constant depending only on $P(X)$ and $\dim(\cO_G(X))$.
 The last part of the lemma follows.

\end{proof}

\begin{remark}  For $G=\SL_2$ and $V=\sl_2$ the adjoint representation, U. Everling prove that the functions $\{f_n(Y)\}$ stabilize
as $n\rightarrow \infty$ by employing similar arguments from which he deduces the existence of $\kappa_G(X, Y)$(cf. \cite[Proposition 4]{Eve98})
in this case. 
\end{remark}

Let $D$ be a subset of $V(\cO_F/\pi^n\cO_F)(n> 0)$. Let $\Phi_D(X)$ be the characteristic function of the preimage of $D$ under the projection $\varphi_n: V(\cO_F)\rightarrow V(\cO_F/\pi^n\cO_F)$.
Let us determine its Fourier transform. If there is no confusion we will also regard $\Phi_D(X)$ as a function on $V(\cO_F/\pi^{d}\cO_F)$ for $d\geq n$.
As for the Fourier transform, we have 
\begin{align}\label{eqn-Fourier-expression}
\hat{\Phi}_D(Y)=\int_{V(F)} \Phi(X)\psi(\langle X, Y\rangle) dX=e_{d}\sum_{X\in V(\cO_F/\pi^{d}\cO_F)}\Phi_D(X)\psi(\langle X, Y\rangle)
\end{align}
with $e_d=\vol(\pi^{d} V(\cO_F))$ and $d\geq \max\{-\v_{\pi}(Y), n\}$.

\begin{proof}[Proof of Theorem \ref{teo-main-results-kernel-function}]
It is clear from the Lemma \ref{lemma-kernel-constancy-well-define}  that the function $\kappa_G(\nu_s, X, Y)$ is well defined.
It remains to show
right hand side of (\ref{teo-main-equality-Fourier}) is well defined as an absolute convergent integral
for $\Re(s)>0$ and that the two sides of (\ref{teo-main-equality-Fourier}) 
is equal.

Let $d$ be an integer.
Let $W_d(X)=\pi^{-d}V(\cO_F)\cap \sO_G(X)$. Since $\hat{\Phi}_D(Y)$ is compactly supported, we have
\begin{align*}
\Orb_G(X, \hat{\Phi}_D, \nu_s)&=\int_{W_{d}}\hat{\Phi}_D(Z)|P(Z)|^sdZ\\
                 &=e_{d}\int_{W_d}(\sum_{Y\in V(\cO_F/\pi^{d}\cO_F)}\Phi_D(Y)\psi(\langle Y, Z\rangle)|P(Z)|^sdZ\\
                 &=e_{d}\sum_{Y\in V(\cO_F/\pi^{d}\cO_F)}\Phi_D(Y)\int_{W_d(X)}\psi(\langle Y, Z\rangle)|P(Z)|^sdZ\\
                 &=e_{d}\sum_{Y\in V(\cO_F/\pi^{d}\cO_F)}\Phi_D(Y)I_d(Y).
\end{align*}
Let $D=\{a_1, \cdots, a_r\}$ be a subset of $V(\cO_F/\pi^n\cO_F)$. If $D$ does not contain 0, then $\Phi_D(Y)\neq 0$ imply that 
$1\leq \v_{\pi}(Y)\leq n$. Following Lemma \ref{lemma-kernel-constancy-well-define}, 
take $d\geq \max_{Y\in D}\{ \v_\pi(Y)\}+3$, we have 
\[
I_\infty(Y)=I_d(Y)=\kappa_G(\nu_s, X, Y).
\]
Hence 
\[
\Orb_G(X, \hat{\Phi}_D, ,\nu_s)=e_{d}\sum_{Y\in V(\cO_F/\pi^{d}\cO_F)}\Phi_D(Y)I_\infty(Y)=\int_{V(F)}\Phi_D(Y)\kappa_G(\nu_s, X, Y)dY.
\]
It remains to treat the case where $D=\{0\}$, i.e., 
\[
\Phi_D(X)=1_{\pi^nV(\cO_F)}(X).
\]
Let $\delta_1$ be the modulus character of $H(F)$ with respect to the Haar measure on $V$. In this case, we have
\begin{align*}
&\int_{V(F)}\Phi_D(Y)\kappa_G(\nu_s, X, Y)dY\\
&=\int_{\pi^nV(\cO_F)}\kappa_G(\nu_s, X,  Y) dY\\
&=\delta_1(\pi)^n\int_{V(\cO_F)}\kappa_G(\nu_s, X, \pi^n Y) dY\\
&=\sum_{j=0}^{\infty} \delta_1(\pi)^{n+j}\int_{V(\cO_F)\backslash \pi V(\cO_F)}\kappa_G(\nu_s, X, \pi^{n+j} Y) dY,
\end{align*}
and 
\begin{align*}
&\kappa_G(\nu_s, X, \pi^{n+j} Y) \\
&=\int_{\sO_G(X)}|P(Z)|^s\psi(\langle Z, \pi^{n+j} Y\rangle)dZ\\
&=(\delta\nu_{s}(\pi))^{-n-j}\int_{\sO_G(\pi^{n+j}X)}|P(Z)|^s\psi(\langle Z,  Y\rangle)dZ.
\end{align*}
In particular, 
\[
\kappa_G(\nu_s, X, \pi^{n+j} Y)=(\delta\nu_{s}(\pi))^{-n-j}\kappa_G(\nu_s, \pi^{n+j} X,  Y).
\]
Hence, 
\begin{align*}
&\int_{V(F)}\Phi_D(Y)\kappa_G(\nu_s, X, Y)dY\\
&=\sum_{j=0}^{\infty} (\delta\delta_1^{-1}\nu_s(\pi))^{-n-j}\int_{V(\cO_F)\backslash \pi V(\cO_F)}\kappa_G(\nu_s, \pi^{n+j} X, Y) dY.
\end{align*}
This sequence( with $s$ as a complex variable) converges on certain half plane of the complex numbers depending on $|\nu(\pi)|$.
 
For $Y\in V(\cO_F)\backslash \pi V(\cO_F)$, we have $\v_{\pi}(Y)=0$, therefore for $d\geq 3$ we have
\[
\kappa_G(\nu_s, X, Y) =I_{\infty}(Y)=I_{d}(Y)=\int_{W_d(X)}|P(Z)|^s\psi(\langle Z,  Y\rangle)dZ.
\]
Now it comes to the point that we are allowed to interchange the two integrals $\int_{V(\cO_F)\backslash \pi V(\cO_F)}\int_{W_d(\pi^{-n-j}X)}$ by compactness and obtain
\begin{align*}
&\int_{V(F)}\Phi_D(Y)\kappa_G(\nu_s, X, Y)dY\\
&=\sum_{j=0}^{\infty} (\delta\delta_1^{-1}\nu_s(\pi))^{-n-j}\int_{W_d(\pi^{n+j} X)}|P(Z)|^sdZ\int_{V(\cO_F)\backslash \pi V(\cO_F)}\psi(\langle Z,  Y\rangle)dY.
\end{align*}
Now reverse the above process, we obtain 
\begin{align*}
&\int_{V(F)}\Phi_D(Y)\kappa_G(\nu_s, X, Y)dY\\
&=\sum_{j=0}^{\infty} (\delta\delta_1^{-1}\nu_s(\pi))^{-n-j}\int_{W_d(\pi^{n+j} X)}|P(Z)|^sdZ\int_{V(\cO_F)\backslash \pi V(\cO_F)}\psi(\langle Z,  Y\rangle)dY\\
&=\sum_{j=0}^{\infty}\delta_1(\pi)^{n+j}\int_{W_{n+j+d}(X)}|P(Z)|^sdZ\int_{V(\cO_F)\backslash \pi V(\cO_F)}\psi(\langle \pi^{n+j}Z,  Y\rangle)dY.
\end{align*}
We want to take take limit with respect to $d$ but need some justification. Let 
\[
S_d=\sum_{j=0}^{\infty}\delta_1(\pi)^{n+j}\int_{W_{n+j+d}(X)}|P(Z)|^sdZ\int_{V(\cO_F)\backslash \pi V(\cO_F)}\psi(\langle \pi^{n+j}Z,  Y\rangle)dY.
\]
Then for $d_1>d$,
\begin{align*}
&S_{d_1}-S_{d}\\
&=\sum_{j=0}^{\infty}\delta_1(\pi)^{n+j}\int_{W_{n+j+d_1}(X)\backslash W_{n+j+d}(X)}|P(Z)|^sdZ\int_{V(\cO_F)\backslash \pi V(\cO_F)}\psi(\langle \pi^{n+j}Z,  Y\rangle)dY\\
&=\sum_{j=0}^{\infty}\delta_1(\pi)^{n+j}\sum_{k=d}^{d_1}\int_{W_{n+j+k+1}(X)\backslash W_{n+j+k}(X)}|P(Z)|^sdZ\int_{V(\cO_F)\backslash \pi V(\cO_F)}\psi(\langle \pi^{n+j}Z,  Y\rangle)dY\\
&=\sum_{j=0}^{\infty}(\delta\delta_1^{-1}\nu_s(\pi))^{-n-j}\sum_{k=d}^{d_1}\int_{W_{k+1}(\pi^{n+j}X)\backslash W_{k}(\pi^{n+j}X)}|P(Z)|^sdZ\int_{V(\cO_F)\backslash \pi V(\cO_F)}\psi(\langle Z,  Y\rangle)dY.
\end{align*}
From (\ref{eqn-fundamental-vanishing}) and the fact that $\nu_\pi(Y)=0$,
we deduce that for $k\geq 3$, 
\begin{align*}
&\int_{W_{k+1}(\pi^{n+j}X)\backslash W_{k}(\pi^{n+j}X)}|P(Z)|^sdZ\int_{V(\cO_F)\backslash \pi V(\cO_F)}\psi(\langle Z,  Y\rangle)dY\\
&=\int_{V(\cO_F)\backslash \pi V(\cO_F)}(\int_{W_{k+1}(\pi^{n+j}X)\backslash W_{k}(\pi^{n+j}X)}|P(Z)|^s\psi(\langle Z,  Y\rangle)dZ)dY\\
&=0.
\end{align*}

Hence 
\[
S_{d_1}=S_d, \quad \forall d_1>d\geq 3.
\]
Letting $d\rightarrow \infty$ we obtain 
\begin{align*}
&\int_{V(F)}\Phi_D(Y)\kappa(\nu_s, X, Y)dY\\
&=\sum_{j=0}^{\infty}\delta_1(\pi)^{n+j}\int_{\sO_G(X)}|P(Z)|^sdZ\int_{V(\cO_F)\backslash \pi V(\cO_F)}\psi(\langle \pi^{n+j}Z,  Y\rangle)dY.
\end{align*}
Note that $\sum_{j=0}^{\infty}\delta_1(\pi)^{n+j}\int_{V(\cO_F)\backslash \pi V(\cO_F)}\psi(\langle \pi^{n+j}Z,  Y\rangle)dY$ converges absolutely to 
$\int_{\pi^nV(\cO_F)} \psi(\langle Z,  Y\rangle)dY$  because $\delta_1(\pi)<1$. Moreover, the integral
\[
\int_{\sO_G(X)}|P(Z)|^s\int_{\pi^n\cO_F} \psi(\langle Z,  Y\rangle)dY
\]
is absolutely convergent for $\rm{Re}(s)>0$ since $\int_{\pi^n\cO_F} \psi(\langle Z,  Y\rangle)dY$ is of compact support in $Z$ (cf. \cite[Lemma 2]{Igu84}).

Therefore applying  Fubini's theorem, we obtain that the left integral
\[
\int_{V(F)}\Phi_D(Y)\kappa(\nu_s, X, Y)dY=\int_{\sO_G(X)}|P(Z)|^s\int_{\pi^n\cO_F} \psi(\langle Z,  Y\rangle)dY.
\]
is absolutely convergent.

This proves Theorem \ref{teo-main-results-kernel-function} for the functions $\Phi_D$. Observe that the functions 
$\Phi_D$ and their translations under $G(F)$ 
form a basis of the space of compactly supported locally constant functions $\cC_{c}^{\infty}(V)$, we observe that the equality (\ref{teo-main-equality-Fourier})
is $G$-equivariant, hence we are done.

\end{proof}

\begin{proof}[Proof of Theorem \ref{teo-main-results-kernel-function-H}]
Let $X\in V(F)$, then
\[
\sO_{H}(X)=F^\times \sO_{G}(X).
\]
Then
\[
\Orb_H(X, f, \chi\nu_s)=\int_{F^\times }\chi(t) \Orb_G(tX, f, \nu_s) dt.
\]

But by Theorem \ref{teo-main-results-kernel-function}, we have 
\[
\Orb_G(X, \hat{f}, \nu_s)=\int_{V(F)} f(Y)\kappa_G(\nu_s, X, Y)dY.
\]
Hence we obtain
\begin{align*}
\Orb_H(X, \hat{f}, \nu_s)&=\int_{F^\times }\chi(t) \Orb_G(tX, \hat{f}, \nu_s) dt\\
&=\int_{F^\times }\int_{V(F)} f(Y)\kappa_G(\nu_s, tX, Y)\chi(t)dY d^\times t\\
&=\int_{V(F)} f(Y)dY\int_{F^\times }\kappa_G(\nu_s, tX, Y)\chi(t) d^\times t\\
&=\int_{V(F)} f(Y)\kappa_H(\chi\nu_s, X, Y) dY.
\end{align*}
Here we can exchange the order of the integrals because $f\in\cC_{c}^{\infty}(V)$.
In particular, as a distribution we have 
\[
\kappa_H(\chi \nu_s, X, Y) =\int_{F^\times }\kappa_G(\nu_s, tX, Y)\chi(t) d^\times t.
\]
\end{proof}

\begin{remark}  We will call the function $\kappa_H(\nu_s, X, Y)$ and orbital Gauss function.
In special cases, this function has been studied by T. Taniguchi and F. Thorne \cite{TaF13}.
\end{remark} 

Let us consider some examples. For $H=GL_1$ and $V=F$. Then $H$ acts on $V$ by scalar multiplication. 
Let $P(X)=X$.
There are two orbits in this case: $V_0=\{0\}$ and $V_1=F^\times$. The second orbit is open.  

Now let us consider the sequence of functions $f_n=1_{p^{-n}\Z_p}$ then the sequence 
$\{f_n\}_{n=1}^{\infty}$ converge to $1$ in the space of distributions $\cD(V)$ on $V$.  Now for $X\in V_1$, 
in $\cD(V)$ we have 
\begin{align}\label{eqn-computation-orbital-integral-GL_1}
\kappa_H(\nu_s, X, Y)&=\int_{F^\times} |Z|^s\psi(ZY)dZ\\
&=\lim_{n} \int_{F^\times} |Z|^sf_n(Z)\psi(ZY)dZ\\
&=|Y|^{-s-1}\gamma_p(-s)
\end{align}
 where 
 \[
 \gamma_p(s)=\frac{L_p(1-s)}{L_p(s)}=\frac{1-p^{-s}}{1-p^{s-1}}.
 \]
 Similar idea was used by Beineke and Bump in computing the Fourier transform of $g\mapsto |\det(g)|^s$(\cite[\S 6]{Bump06} ).
 This also follows from the functional equation for $H=\GL_1$. In fact consider the zeta integral for the trivial representation we have
\[
Z(\Phi, s)=\int_{F^\times}\Phi(X) |X|^{s-1}dX,
\]
now local function equation tells us
\[
Z(\hat{\Phi}, 1-s)=\gamma(s)Z(\Phi, s).
\]
Note that the definition of $\kappa_H(\nu_s, X, Y)$ extends to $Y=0$. Note that in this case if $s=0$ the function $\kappa_H(\nu_s, X, Y)=0$, hence our 
orbital integral will vanish identically(in distributional sense).

 Generally, for $H=GL_n\times GL_n$ and $V=\gl_n$ with the action
 \[
 ((g_1, g_2), X)\mapsto g_1Xg_2^{-1}.
 \]
 We take $P(X)=\det(X)$. We have finitely many orbits, indexing by the rank of $X$.
 Let us consider $X=\Id$, a point on the open orbit. As before let us take $f_n=1_{p^{-n}\gl_n(\Z_p)}$, then
 $f_n\rightarrow1(n\rightarrow \infty)$ in distribution sense. Let us consider the function
 \begin{align*}
 \kappa_H(\nu_s, X, Y)&=\int_{GL_n(F)} |\det(Z)|^s\psi(\Tr(ZY))dZ\\
 &=\lim_{n}\int_{GL_n(F)} f_n(X)|\det(X)|^s\psi(\Tr(ZY))dZ.
 \end{align*}
 As before, our computation shows that 
 \[
 \kappa_H(\nu_s, X, Y)=|\det(Y)|^{-s-n}\gamma(-s), \quad \gamma(s)=\frac{L_p(n-s)}{L_p(s)},
 \]
 where $L_p(s)=\frac{1}{1-p^{-s}}$.
 As before, this also follows from the local functional equation for the trivial representation of $\GL_n$.
 \\
 \section{Application of the Main Results}
 
Let us explain how to derive the theorem of Harish-Chandra from our main results. 
Let $G$ be a semi-simple group and $\kg$ its Lie algebra. We consider the adjoint representation
\[
\Ad: G\rightarrow \Aut(\kg).
\]
In particular, we have $H=\Ad(G)\G_m$, where $\G_m$ is regarded as the center of $\Aut(\kg)$.
We also fix an $G$-invariant symmetric bilinear form $\langle , \rangle$ on $\kg$.
When $\kg$ is a classical Lie algebra of type $A, B, C, D$, which can be realized in a standard way as 
a sub Lie algebra of the Lie algebra $\gl_n$(cf. \cite[Chapter 2]{Hum12}). Then $\langle , \rangle$ is simply the restriction of 
the trace pairing
\[
(X, Y)\mapsto \Tr(XY).
\]
Furthermore, following \cite[Page 2]{De99}, we consider the invariant polynomial $P(X)$ as the coefficient of $t^\ell$ in 
\[
\det(t-\ad(X)), \quad \ell=\rank(\kg).
\]
 In general the algebra of invariants $S(\kg)^G$ is identified by Chevalley with $S(\kg)^W$ with $W$ the Weyl group.
 Note that by a theorem of Deligne and Rao \cite[Theorem 1]{Rao72}, the integral 
 \[
 \Orb_G(Y, f)=\int_{\sO_G(Y)}f(X)dX
 \]
 is well defined for any $f\in \cC_c^\infty(\kg)$ and $X\in\{X\in \kg: P(X)\neq 0\}$. As a consequence, we can take $s=0$ in the 
 equation (\ref{teo-main-equality-Fourier}) and recover the theorem of Harish-Chandra.
 
\begin{remark}  Similar considerations can be applied to the case of representability for the Jacquet–Rallis orbital integral in \cite{Zhang12} once the analogue 
 of theorem of Deligne and Rao is established, which is an interesting question.
 \end{remark}

 \par \vskip 1pc
{\bf Acknowledgement.}
The work is done when the author is a postdoc at Yau Mathematical Sciences Center
of Tsinghua University. He wants to thank their hospitality. He also would like to thank Bin Xu for a careful reading of the first version of the paper.
He also would like to thank the referee for pointing out a mistake in using a theorem of Igusa on rationality of orbital integrals.

\bibliographystyle{plain}
\bibliography{biblio}


\end{document}